\documentclass[12pt]{article}
\usepackage{amsfonts,amsmath,latexsym,amssymb,euscript,graphicx,mathrsfs}
\usepackage{mathtools}
\usepackage{mathabx}
\usepackage{lastpage} 
\usepackage{graphicx} 
\usepackage{amsthm}
\usepackage[margin=1in]{geometry}
\usepackage{amsthm,amssymb}

\usepackage{enumerate}
\usepackage{wrapfig}
\usepackage{filecontents}
\usepackage{hyperref}
\usepackage{bookmark}
\usepackage{bbm} 
\DeclareMathAlphabet\EuFrak{U}{euf}{m}{n}

\usepackage{euscript}

\usepackage[english]{babel}
\usepackage[dvipsnames]{xcolor}
 
\newtheorem{prop}{Proposition}[section]

\newtheorem{lemme}[prop]{Lemma}

\newtheorem{thm}[prop]{Theorem}

\numberwithin{equation}{section}

\usepackage{array}
\bibliography{Ma-Nu-citations}

\def\HH{\EuFrak H}
\def\1{{\mathbbm{1}}}

\newcommand{\msplitsp}{\phantom{=}\hspace{0.5em}}

\renewcommand{\geq}{\geqslant}
\def\leq{\leqslant}

\usepackage{epstopdf}

\DeclarePairedDelimiter{\floor}{\lfloor}{\rfloor}

\numberwithin{subcase}{case}

\begin{document}

\title{Rate of convergence for the weighted Hermite variations of the fractional Brownian motion} 
%
%
%
%
%

\title{Rate of convergence for the weighted Hermite variations of the fractional Brownian motion}
\author{Nicholas Ma  \footnote{Email:\texttt{njma@ku.edu}}  and David Nualart 
 \footnote{Email:\texttt{nualart@ku.edu}}
 \thanks{%
The work of D. Nualart is supported by the NSF Grant DMS 1811181} \\
Department of Mathematics\\
The University of Kansas\\
Lawrence, Kansas 66045, USA
 \date{}}
\maketitle

\begin{abstract}
  In this paper we obtain a rate of convergence in the central limit theorem for high order weighted Hermite variations of the fractional Brownian motion. The proof is based on the techniques of Malliavin calculus and the quantitative stable limit theorems proved by Nourdin, Nualart and Peccati in \cite{nourdin2016}.
   \newline


\end{abstract}

\maketitle


\section{Introduction}
The fractional Brownian motion $B= \{B_t, t\ge 0\}$  is characterized by  being a zero-mean Gaussian  self-similar process  with stationary increments and variance $E(B_t^2)=t^{2H}$.  The self-similarity index $H \in (0,1)$ is called the Hurst parameter.   The fractional Brownian motion was first introduced by Kolmogorov in 1940. However, the landmark paper by Mandelbrot and Van Ness \cite{fbm} gave fractional Brownian motion its name and inspired much of the modern literature on the subject. 

\medskip
The study of single path behavior of stochastic processes often uses their power variations.  In particular, the fractional Brownian motion is known to have a $1/H$-variation on any finite time interval equal to the length of the interval multiplied by the constant $\kappa_H=E[|Z^{1/H}|]$, where $Z$ is a $N(0,1)$ random variable. That means, if we consider  the uniform partition of the interval $[0,1]$ into $n\ge 1$ intervals and for $0\le k\le n-1$  we denote $\Delta B_{k/n}=B_{(k+1)/n}-B_{k/n}$, we have
\[
\lim_{n\rightarrow \infty}\sum_{k=0}^{n-1}   |\Delta B_{k/n}|^{1/H}  \rightarrow  \kappa_H,
\]
where the convergence  holds almost surely and in $L^p(\Omega) $ for any $p\ge 2$.  A central limit theorem associated with this approximation can be obtained  by expanding the function $|x|^{1/H}$ into Hermite polynomials. In particular,  as a consequence of the Breuer-Major theorem \cite{BM}, for each integer $q\ge 2$ such that
$H<1-\frac 1{2q}$, we have the convergence in law
\[
\lim_{n\rightarrow \infty} \sum_{k=0}^{n-1}    \frac1{\sqrt n} H_q( n^{H} \Delta B_{k/n})    \stackrel{\mathcal{L}} {\rightarrow}   N(0,\sigma_{H,q}^2),
\]
where $H_q$ is the $q$th Hermite polynomial and 
\begin{equation} \label{sig}
\sigma_{H,q}^2=q!  \sum_{k \in \mathbb{Z}} \rho_H(k) ^q.
\end{equation}
Here
\begin{equation}
\rho_H(k)=\frac12(|k+1|^{2H} + |k-1|^{2H}-2|k|^{2H}),\qquad k \in\mathbb{Z}\label{eqrh}
\end{equation}
denotes the covariance of the stationary sequence $\{ B_{k+1}-B_k,  k\ge 0\}$.
 
 \medskip
There has been intensive research on the asymptotic behavior of the {\it weighted Hermite variations} of the fractional Brownian motion $B$, defined by
\begin{equation}
F_n =  \frac 1{\sqrt{n}}  \sum_{k=0}^{n-1} f(B_{k/n}) H_q(n^H \Delta B_{k/n})\label{eqfn},
\end{equation}
where $f$ is a given function.
The analysis of the asymptotic behavior of these quantities is motivated, for instance, by the study of the exact rates of convergence of some approximation schemes of scalar stochastic differential equations driven by the fractional Brownian motion  (see, for instance, \cite{GN,NN}), in addition to  the traditional applications of  $q$-variations to parameter estimation problems. {Additionally, as discovered in \cite{nourdin2008asymptotic}, new phenomena arise in the presence of weights. These phenomena were more fully studied in \cite{nourdin2010wpv} and  \cite{nourdin2009asymptotic}.}

\medskip
{It was shown by Nourdin, Nualart, and Tudor in \cite{nourdin2010wpv}} that, when $\frac1{2q}<H<1-\frac1{2q}$, the sequence $F_n$ defined in \eqref{eqfn}  converges in law to a mixture of Gaussian distributions.   More precisely, the following stable convergence holds as $n$ tends to infinity
\begin{equation} \label{conv}
( B, F_n)   \stackrel{\mathcal{L}} {\rightarrow} \left( B, \sigma_{H,q} \int_0^1 f(B_s) dW_s \right),
\end{equation}
where $W= \{W_t, t\in [0,1]\}$ is a standard Brownian motion independent of $B$, and $\sigma_{H,q}$ is given by \eqref{sig}.
\medskip
For $H$  outside the interval $( \frac1{2q}, 1-\frac1{2q})$ different phenomena occur. Specifically, it was shown in \cite{nourdin2010central} that when $0<H<\frac1{2q}$, $n^{qH-1/2} F_n$ converges in $L^2(\Omega)$ to $(-2)^{-q} \int_0^1 f^{(q)}(B_s)\,ds$, and when $1-\frac1{2q}<H<1$, $n^{q(1-H) -1/2} F_n$ converges in $L^2(\Omega)$ to $ \int_0^1 f(B_s)\,dZ_s^{(q)}$ where $Z^{q}$ is the Hermite process. In the critical case $H=\frac1{2q}$, there is convergence in law to a linear combination of the $H<\frac1{2q}$ and $\frac1{2q}<H<1-\frac1{2q}$ cases, and in the critical case $H=1-\frac1{2q}$ there is convergence in law with an additional logarithmic factor (see \cite{nourdin2010central}).

\medskip
We recall the assumption $\frac1{2q}<H<1-\frac1{2q}$,  under which convergence of the sequence $F_n$ was shown in \cite{nourdin2010wpv}. A natural question is to study the rate of the convergence in law  of the sequence $F_n$ to  $\sigma_{H,q} \int_0^1 f(B_s) dW_s$ stated in (\ref{conv}). When $f\equiv 1$,  Stein's method, combined with Malliavin calculus, allows one to derive  upper bounds for the rate of convergence of  the total variation distance (see the monograph by Nourdin and Peccati \cite{nourdin2012normal} and the references therein). {The series of papers \cite{binotto,nourdin2008asymptotic,nourdin2010central,nourdin2016,nourdin2010wpv,nourdin2009asymptotic,nourdin2010weak} have greatly contributed to the development of the Mallavin-stein approach, which has become a powerful and general tool to study limit theorems for functionals of Gaussian processes. }In the case of weighted variations, this methodology is no longer applicable.  In \cite{nourdin2016}, Noudin, Nualart, and Peccati developed a  new approach based on the interpolation method that provides
quantitative rates for the convergence of multiple Skorohod integrals to  a  mixture of Gaussian laws.
A basic result in this direction is Proposition \ref{p1} below.  In  \cite{nourdin2016}, the authors apply this approach to deduce a rate of convergence for $F_n$ in the case $q=2$ and $\frac14<H<\frac34$.

\medskip
The main purpose of this paper is to apply  the technique introduced in \cite{nourdin2016}   to weighted Hermite variations of any order $q\ge 2$, extending the results proved for  weighted quadratic variations.
We will show that the rate of convergence is bounded, up to a constant, by $n^{\Phi(H)}$, where the exponent $\phi(H)$ is defined by
\begin{equation}\label{eqphi}
\phi(H) = \left( \left | H-\frac 12 \right | - \frac 12 \right) \vee \left( q \left  | H-\frac 12 \right | - \frac {q-1}2 \right).
\end{equation}
That is, 
\[
\phi(H)=
\begin{cases}
  (-H)  \vee ( -qH +\frac 12)  & {\rm if}  \,\,  H \le \frac 12, \\
  (H-1)  \vee ( q(H-1) +\frac 12) &   {\rm   if} \,\,  H>\frac 12.
 \end{cases}
 \]
Notice that $\phi(H)=0 $ when $H$ is equal to one of the end points of the interval $( \frac1{2q}, 1-\frac1{2q})$ and it is symmetric with respect to the middle point $\frac 12$.  Moreover, there are unexpected transition phases
when $H=\frac 1{2q-2}$ and when $H= 1-\frac 1{2q-2}$.

\medskip
In order to state our main result, we need some notation and definitions.
We say that a function $f:\mathbb{R}\to \mathbb{R}$ has \emph{moderate growth} if there exist positive constants $A$, $B$, and $\alpha<2$ such that for all $x\in \mathbb{R}$, $|f(x)|\le A\exp(B|x|^\alpha)$.

\medskip
Given a measurable function $f: \mathbb{R} \to \mathbb{R}$, an integer $N\ge 0$ and a real number $p\ge 1$, we define the semi-norm
\begin{equation}\label{norm}
\| f\|_{N,p} = \sum_{i=0}^N \sup_{0\le t \le 1} \|f^{(i)}\|_{L^p(\mathbb{R},\gamma_t)}
\end{equation}
where $\gamma_t$ is the normal distribution $N(0,t)$.

\medskip 
We can now state the main result of this paper. This result extends the work done in \cite{nourdin2010wpv} that proves stable convergence for any $q$, and the  work done in 
\cite{nourdin2016} that provides a quantitative bound in the $q=2$ case.

\begin{thm}\label{mainresult}
Let $q\in \mathbb N$, $q\ge 2$. Assume that the Hurst index $H$ of  the fractional Brownian motion $B$ belongs to $\left( \frac1{2q},1-\frac1{2q}\right).$ Consider a function $f:\mathbb R\to \mathbb R$ of class $C^{2q}$ such that $f$ and its first $2q$ derivatives have moderate growth. Suppose in addition that 
\[
E\left[\left(\int_0^1 f^2(B_s)\,ds\right)^{(1-q)\alpha}\right]<\infty
\]
 for some $\alpha>1$. Consider the sequence of random variables $F_n$ defined by \eqref{eqfn}. Set $S=\sqrt{\sigma_{H,q}^2\int_0^1 f^2(B_s)\,ds}$, where $\sigma_{H,q}$ is given in  (\ref{sig}). Then, for any function $\varphi:\mathbb R \to \mathbb R$ of class $C^{2q+1}$ with $\| \varphi^{(k)}\|_\infty <\infty$ for each $k=0,\ldots, 2q+1$, we have
\begin{align}
\begin{split}
\left| E[\varphi(F_n)]-E[\varphi(S\eta)] \right| &\le  C_{H,f,q}\sup_{1\le i\le 2q+1}\|\varphi^{(i)}\|_{\infty} n^{\phi(H)},
\end{split}
\end{align}
where $\eta$ is a standard normal variable independent of $B$. The constant $C_{H,f,q}$ has the form 
\[
C_{H,f,q} = C_{H,q} \max\left\{\|f\|_{2q,2}, E[S^{(2-2q)\alpha}]^{1/\alpha}\|f\|_{2q,(2q+2)\beta }^{2q+1}  \right\},
\]
 and $1/\alpha+1/\beta=1$.
\end{thm}

The paper is organized as follows. Section 2 contains some preliminaries on the fractional Brownian motion and  its associated Malliavin calculus. The basic rate of convergence result for multiple Skorohod integrals, Proposition \ref{p1},  is also stated in this section. Section 3 is devoted to the proof of Theorem \ref{mainresult}. The proof intensively uses the techniques of Malliavin calculus  and detailed estimates for the sums of powers of the covariance function $\rho_H$ obtained in Section 2 (see Lemma  \ref{threebeta1}).  {   We have included an example at the end of Section 3  that explains why the phase of transition in the rate of convergence occurs.
} A technical lemma is proved in the Appendix.

\section{Preliminaries}

In this section we  first present some definitions and  basic results on the factional Brownian motion and  the associated Malliavin calculus.  The reader is referred to the monographs \cite{nualart2006malliavin} and \cite{nourdin2012normal} for a detailed account on these topics.  We also recall an upper bound for the approximation of multiple Skorohod integrals by a mixture of Gaussian laws that will play a fundamental role in the proof of the main result.

\subsection{Fractional Brownian motion}

Consider a fractional Brownian motion $B=\left\{B_t,  t\in[0,1]\right\}$ with Hurst parameter $H\in(0,1)$ defined in a probability space $(\Omega, \mathcal{F}, P)$. That means, $B$ is a zero mean Gaussian process with covariance 
\[
E(B_tB_s)=\frac12(t^{2H}+s^{2H}-|t-s|^{2H}), \quad s,t\in [0,1].
\]
Let $\mathcal{E}$ be the space of step functions on $[0,1]$ and consider the Hilbert space defined as the  closure of $\mathcal{E}$ under the inner product
$\langle \1_{[0,t]}, \1_{[0,s]}\rangle_{\mathfrak{H}}=E(B_tB_s)$ for $s,t\in[0,1]$.  Then the mapping $\1_{[0,t]} \rightarrow B_t$ can be extended to a linear isometry between $\HH$ and the Gaussian space generated by $B$. We denote by $B(h)$ the image of $h\in \HH$ by this isometry. With this notation, $\{B(h), h \in \HH\}$ is an {\it isonormal Gaussian process} associated with the Hilbert space $\HH$.
 We refer the reader to the references \cite{nourdin2012selected,nualart2006malliavin}  for  a detailed study of  this process.
  
For any integer $q\geq 1$, we denote by $%
\EuFrak H^{\otimes q}$ and $\EuFrak H^{\odot q}$, respectively, the $q$th tensor product and the $q$th symmetric tensor product of $\EuFrak H$.

From now on, we assume that $\mathcal{F}$ is the $P$-completion of the $\sigma$-field generated by $B$. 
For every integer $q\geq 1$, we let $\mathcal{H}_{q}$ be the $q$th {\it Wiener chaos} of $B$,
that is, the closed linear subspace of $L^{2}(\Omega)$
generated by the random variables $\{H_{q}(B(h)),h\in \EuFrak H,\left\|
h\right\| _{\EuFrak H}=1\}$, where $H_{q}$ is the $q$th Hermite polynomial defined by
\[
H_q(x)=(-1)^q e^{x^2/2}\frac{d^q}{dx^q}\big(e^{-x^2/2}\big).
\]
We denote by $\mathcal{H}_{0}$ the space of constant random variables. For
any $q\geq 1$, the mapping $I_{q}(h^{\otimes q})=H_{q}(B(h))$ provides a
linear isometry between $\EuFrak H^{\odot q}$
(equipped with the modified norm $\sqrt{q!}\left\| \cdot \right\| _{\EuFrak %
H^{\otimes q}}$) and $\mathcal{H}_{q}$ (equipped with the $L^2(\Omega)$ norm). For $q=0$, we set by convention $%
\mathcal{H}_{0}=\mathbb{R}$ and $I_{0}$ equal to the identity map.

It is well-known (Wiener chaos expansion) that $L^{2}(\Omega)$
can be decomposed into the infinite orthogonal sum of the spaces $\mathcal{H}%
_{q}$, that is: any square integrable random variable $F\in L^{2}(\Omega)$ admits the following chaotic expansion:
\begin{equation}
F=\sum_{q=0}^{\infty }I_{q}(f_{q}),  \label{E}
\end{equation}%
where $f_{0}=E[F]$, and the $f_{q}\in \EuFrak H^{\odot q}$, $q\geq 1$, are
uniquely determined by $F$. 


Let $\{e_{k},\,k\geq 1\}$ be a complete orthonormal system in $\EuFrak H$.
Given $f\in \EuFrak H^{\odot p}$, $g\in \EuFrak H^{\odot q}$ and $%
r\in\{0,\ldots ,p\wedge q\}$, the $r$th {\it contraction} of $f$ and $g$
is the element of $\EuFrak H^{\otimes (p+q-2r)}$ defined by
\begin{equation}
f\otimes _{r}g=\sum_{i_{1},\ldots ,i_{r}=1}^{\infty }\langle
f,e_{i_{1}}\otimes \ldots \otimes e_{i_{r}}\rangle _{\EuFrak H^{\otimes
r}}\otimes \langle g,e_{i_{1}}\otimes \ldots \otimes e_{i_{r}}\rangle _{%
\EuFrak H^{\otimes r}}.  \label{v2}
\end{equation}%
Notice that $f\otimes _{r}g$ is not necessarily symmetric. We denote its
symmetrization by $f\widetilde{\otimes }_{r}g\in \EuFrak H^{\odot (p+q-2r)} $%
. Moreover, $f\otimes _{0}g=f\otimes g$ equals the tensor product of $f$ and
$g$ while, for $p=q$, $f\otimes _{q}g=\langle f,g\rangle _{\EuFrak %
H^{\otimes q}}$. Contraction operators  appear in the following formula for products of multiple Wiener-It\^o integrals (see, for instance, \cite{nualart2006malliavin} Proposition 1.1.3) :
\begin{equation}\label{prodform}
I_p(f) I_q(g) = \sum_{r=0}^{p\wedge q} r! \binom p r \binom q r I_{p+q-2r} (f \widetilde{\otimes}_r g),
\end{equation}
for any  $f\in \EuFrak H^{\odot p}$ and  $g\in \EuFrak H^{\odot q}$.

\medskip
We consider the uniform partition of the interval $[0,1]$, and, for $n\ge 1$ and $k=0, 1,\ldots, n-1$, let $\delta_{k/n}=\1_{[k/n,(k+1)/n]}$ and $\varepsilon_{k,n}=\1_{[0,k/n]}$.  We will make use of the notation
\begin{equation}  \label{alpha}
\alpha_{k,t} = \langle \delta_{k/n},\1_{[0,t]}\rangle_{\mathfrak{H}}
\end{equation}
and
\begin{equation} \label{beta}
\beta_{j,k} = \langle \delta_{j/n},\delta_{k/n}\rangle_{\mathfrak{H}}
\end{equation}
for any $t\in [0,1]$ and $j,k=1,\dots, n-1$. Notice that
\[
\beta_{j,k} = n^{-2H} \rho_H(j-k)
\]
where $\rho_H$ has been defined in $\eqref{eqrh}$.

For the proof of Theorem \ref{mainresult} we need several technical estimates on the quantities  $\alpha_{k,t} $ and  $\beta_{j,k}$.
We first reproduce a useful technical lemma from \cite{nourdin2016}.
\begin{lemme}
\label{lem1} Let $0<H<1$ and $n\geq 1$. We have, for some constant $C_H$ that depends on $H$,
\begin{itemize}
\item[(a)] $\left|  \alpha_{k,t} \right| \leqslant n^{-(2H\wedge 1)}$ for any $t\in
\lbrack 0,1]$ and $k=0,\dots, n-1$.
\item[(b)] $\sup_{t\in \lbrack 0,1]}\sum_{k=0}^{n-1}\left|  \alpha_{k,t}\right| \leq C_H$.
\end{itemize}
\end{lemme} 

The next lemma estimates the sum of powers of the terms $\beta_{j,k}$.
\begin{lemme}\label{betas}\hspace{0em}
\begin{enumerate}[(a)]
\item For $a\ge 1$ and $0\le i \le n-1$ and some constant $C_H$ that depends on $H$,
\[
\sum_{j=0}^{n-1}\left| \beta_{j,i}\right|^a \le  C_H n^{(1-2a)\vee (-2aH)}.
\]
\item For $a\ge 1$ and for some constant $C_H$ that depends on $H$,
\begin{equation}
\sum_{j,k=0}^{n-1} |\beta_{j,k}|^a \le C_H n^{(2-2a)\vee (1-2aH)}.
\end{equation}
\end{enumerate}
\end{lemme}
\begin{proof}
Using (\ref{eqrh}), we have 
\[
\sum_{j=0}^{n-1} |\beta_{j,i}|^a = n^{-2aH} \sum_{j=0}^{n-1} |\rho_H(j-i)|^{a}.
\]
Taking into account that 
 $|\rho_H(j-i)|^a$ converges to zero as $j$ tends to infinity at the rate  $j^{a(2H-2)}$, when $a(2H-2)<-1$ the above sum is bounded by a constant. When $a(2H-2)\ge -1$, it diverges at the rate $n^{a(2H-2)+1}$. This gives the estimate in part (a). 
For (b), we make the change of indices $(j,k) \rightarrow (j, h)$, where $h=j-k$, we   estimate the sum in $j$ by    $n$ and apply (a) for the sum in $h$.
\end{proof}

We recall  a version for infinite sums of the rank-one Brascamp-Lieb inequality that will be used to estimate sums of products of the terms $\beta_{j,k}$.
The statement is reproduced 
 from \cite[Proposition 2.4]{zhou}, which is taken from the works  \cite{barthe,tao} and \cite{brascamp}:
\begin{prop}[Brascamp-Lieb inequality]\label{brascamplieb}
Let $2\le M\le N$ be fixed integers. Consider nonnegative measurable functions $f_j:\mathbb{R}\to \mathbb{R}_+$, $1\le j\le N$, and fix nonzero vectors $\mathbf{v_j}\in \mathbb{R}^M$. Fix positive numbers $p_j$, $1\le j\le N$, verifying the following conditions:
\begin{enumerate}[(i)]
\item $\sum_{j=1}^N p_j = M$.
\item For any subset $I\subset\{1,\ldots, N\}$, we have $\sum_{j\in I} p_j \le \dim ( \mathrm{Span} \{\mathbf{v_j},j\in I\})$.
\end{enumerate}
Then, there exists a finite constant $C$, depending on $N,M$ and the $p_j$'s such that
\begin{equation}
\sum_{\mathbf{k}\in \mathbb{Z}^M} \prod_{j=1}^N f_j(\mathbf{k}\cdot \mathbf{v_j}) \le C \prod_{j=1}^N \left( \sum_{k\in \mathbb{Z}} f_j(k)^{1/p_j}   \right)^{p_j}.
\end{equation}
\end{prop}

We will use the Brascamp-Lieb inequality to prove the following lemma.
\begin{lemme}\label{threebeta1}
For $a\ge 1$, $b\ge 1$, $\ell = 1,\dots, n-1$, and for some constant $C_H$,
\[
\sum_{j,j'=0}^{n-1} |\beta_{j,\ell}|^a |\beta_{j',\ell}|^a |\beta_{j,j'}|^b\le C_H n^{(-2H(2a+b)) \vee (2-2(2a+b))}.
\]
\end{lemme}
\begin{proof}
We have
\begin{align}
\sum_{j,j'=0}^{n-1} |\beta_{j,\ell}|^a |\beta_{j',\ell}|^a |\beta_{j,j'}|^b = n^{-2H(2a+b)} \sum_{j,j'=0}^{n-1} \left|\rho_H(\ell-j) \rho_H(\ell-j')\right|^a \left|\rho_H(j-j')\right|^b.
\end{align}
Making the substitutions $k_1=\ell-j$ and $k_2=\ell-j'$, we can write the above sum as
\begin{align}
\sum_{k_1,k_2=\ell-n+1}^\ell \left|\rho_H(k_1) \rho_H(k_2)\right|^a \left| \rho_H(k_1-k_2)\right|^b.
\end{align}
Let $N=3$, $M=2$, 
\[
f_1(x)=f_2(x) = |\rho_H(x)|^a \mathbf{1}_{\{\ell-n \le x \le \ell\}}
\]
 and 
 \[
 f_3(x) = |\rho_H(x)|^b \mathbf{1}_{\{|x| \le n-1\}}.
 \]
 Consider the vectors 
  $\mathbf{v}_1=(1,0)$, $\mathbf{v}_2=(0,1)$, and $\mathbf{v}_3=(1,-1)$. Applying
Proposition \ref{brascamplieb},  we have
\begin{eqnarray*}
&& \sum_{k_1,k_2=\ell-n+1}^\ell \left|\rho_H(k_1) \rho_H(k_2)\right|^a \left| \rho_H(k_1-k_2)\right|^b 
\\ &&  \qquad  \qquad \le C \left(\sum_{k\in \mathbb{Z}} f_1(k)^{1/p_1} \right)^{p_1}\left(\sum_{k\in \mathbb{Z}} f_2(k)^{1/p_2} \right)^{p_2}\left( \sum_{k\in \mathbb{Z}} f_3(k)^{1/p_3}\right)^{p_3}\\
&&  \qquad  \qquad=C \left(\hspace{1.1em}\sum_{\mathclap{k=\ell-n+1}}^\ell |\rho_H(k)|^{a/p_1} \right)^{p_1}\!\!\left(\hspace{1.1em}\sum_{\mathclap{k=\ell-n+1}}^\ell |\rho_H(k)|^{a/p_2} \right)^{p_2}\!\!\left(\hspace{1em} \sum_{\mathclap{k=-(n-1)}}^{n-1} |\rho_H(k)|^{b/p_3}\right)^{p_3}.
\end{eqnarray*}
The choices $p_1=p_2=2a/(2a+b)$ and  $p_3=2b/(2a+b)$ satisfy the conditions of  Proposition \ref{brascamplieb}.  Note that $p_1+p_2+p_3 = 2$ and $a/p_1=a/p_2=b/p_3=(2a+b)/2$. In this way, we can write 
\begin{eqnarray*}
 \sum_{k_1,k_2=\ell-n+1}^\ell \left|\rho_H(k_1) \rho_H(k_2)\right|^a \left| \rho_H(k_1-k_2)\right|^b  
& \le & C  \left(\sum_{|k| \le 2 n} | \rho_H(k) | ^{ (2a+b)/2} \right)^2 \\
&\le& C n^{ 2[(2H-2)(2a+b)/2 +1)_+]},
 \end{eqnarray*}
 which implies the desired estimate.
\end{proof}

\subsection{Malliavin calculus}

Let us now introduce some elements of the Malliavin calculus  with respect
to the fractional Brownian motion $B$. Let $\mathcal{S}$
be the set of all smooth and cylindrical random variables of
the form
\begin{equation}
F=g\left( B(\phi _{1}),\ldots ,B(\phi _{n})\right) ,  \label{v3}
\end{equation}%
where $n\geq 1$, $g:\mathbb{R}^{n}\rightarrow \mathbb{R}$ is an infinitely
differentiable function with compact support, and $\phi _{i}\in \EuFrak H$.
The derivative  of $F$ with respect to $B$ is the element of $%
L^{2}(\Omega ;\EuFrak H)$ defined as
\begin{equation*}
DF\;=\;\sum_{i=1}^{n}\frac{\partial g}{\partial x_{i}}\left( B(\phi
_{1}),\ldots ,B(\phi _{n})\right) \phi _{i}.
\end{equation*}
By iteration, one can
define the $q$th derivative $D^{q}F$ for every integer $q\geq 2$, with $D^{q}F\in L^{2}(\Omega ;%
\EuFrak H^{\odot q})$.
For integers $q\geq 1$ and real numbers  $p\geq 1$, the Sobolev space ${\mathbb{D}}^{q,p}$ is defined as the closure of $\mathcal{S}$ with respect to the norm $\Vert \cdot \Vert_{\mathbb{D}^{q,p}}$, defined by
the relation
\begin{equation*}
\Vert F\Vert _{\mathbb{D}^{q,p}}^{p}\;=\;E\left[ |F|^{p}\right] +\sum_{i=1}^{q}E\left(
\Vert D^{i}F\Vert _{\EuFrak H^{\otimes i}}^{p}\right) .
\end{equation*}

The derivative  operator $D$ verifies the following chain rule. If $%
\varphi :\mathbb{R}^{n}\rightarrow \mathbb{R}$ is continuously
differentiable with bounded partial derivatives and if $F=(F_{1},\ldots
,F_{n})$ is a vector of elements of ${\mathbb{D}}^{1,2}$, then $\varphi
(F)\in {\mathbb{D}}^{1,2}$ and
\begin{equation*}
D(\varphi (F))=\sum_{i=1}^{n}\frac{\partial \varphi }{\partial x_{i}}%
(F)DF_{i}.
\end{equation*}

We denote by $\delta $ the adjoint of the operator $D$, also called the
{\it divergence operator} or {\it Skorohod integral} (see, e.g., \cite[Section 1.3.2]{nualart2006malliavin} for an explanation of this terminology).
A random element $u\in L^{2}(\Omega ;\EuFrak %
H)$ belongs to the domain of $\delta $, denoted by $\mathrm{Dom}\delta $, if and
only if it verifies
\begin{equation*}
\big|E\big(\langle DF,u\rangle _{\EuFrak H}\big)\big|\leq c_{u}\,\sqrt{E(F^2)}
\end{equation*}%
for any $F\in \mathbb{D}^{1,2}$, where $c_{u}$ is a constant depending only
on $u$. If $u\in \mathrm{Dom}\delta $, then the random variable $\delta (u)$
is defined by the duality relationship (called `integration by parts
formula'):
\begin{equation}
E(F\delta (u))=E\big(\langle DF,u\rangle _{\EuFrak H}\big),  \label{ipp}
\end{equation}%
which holds for every $F\in {\mathbb{D}}^{1,2}$. Formula (\ref%
{ipp}) extends to the multiple Skorohod integral $\delta ^{q}$, and we
have
\begin{equation}
E\left( F\delta ^{q}(u)\right) =E\left( \left\langle D^{q}F,u\right\rangle _{%
\EuFrak H^{\otimes q}}\right),  \label{dual}
\end{equation}%
for any element $u$ in the domain of $\delta ^{q}$ and any random variable $%
F\in \mathbb{D}^{q,2}$. Moreover, $\delta ^{q}(h)=I_{q}(h)$ for any $h\in %
\EuFrak H^{\odot q}$.

The following statement will be used in the paper, and is proved in \cite{nourdin2010central}.
\begin{lemme}\label{lemme}
Let $q\geq 1$ be an integer. Suppose that  $F\in {\mathbb{D}}^{q,2}$, and
let $u$ be a symmetric element in $\mathrm{Dom}\delta ^{q}$. Assume that,
for any $\ 0\leq r+j\leq q$,
$\left\langle D^{r}F,\delta ^{j}(u)\right\rangle _{\EuFrak H^{\otimes r}}\in
L^{2}(\Omega ;\EuFrak H^{\otimes q-r-j})$. Then, for any  $r=0,\ldots ,q-1$, $%
\left\langle D^{r}F,u\right\rangle _{\EuFrak H^{\otimes r}}$ belongs to the
domain of $\delta ^{q-r}$ and we have
\begin{equation}
F\delta ^{q}(u)=\sum_{r=0}^{q}\binom{q}{r}\delta ^{q-r}\left( \left\langle
D^{r}F,u\right\rangle _{\EuFrak H^{\otimes r}}\right),  \label{t3}
\end{equation}
with the convention that $\delta^0(v)=v$, $v\in L^2(\Omega)$, and $D^0F=F$, $F\in L^2(\Omega)$.
\end{lemme}

For any Hilbert space $V$, we denote by $\mathbb{D}^{k,p}(V)$ the
corresponding Sobolev space of $V$-valued random variables (see \cite[page 31]{nualart2006malliavin}%
). The operator $\delta^q$   is continuous from $\mathbb{D}^{k,p}(\EuFrak H^{\otimes q})$
to $\mathbb{D}^{k-q,p}$, for any $p>1$ and  any integers $k\ge q\ge 1$, that is,
we have%
\begin{equation}
\left\| \delta^q (u)\right\| _{\mathbb{D}^{k-q,p}}\leq c_{k,p}\left\| u\right\| _{\mathbb{D}%
^{k,p}(\EuFrak H^{\otimes q})},  \label{Me2}
\end{equation}%
for all $u\in\mathbb{D}^{k,p}(\EuFrak H^{\otimes q})$, and for some constant $c_{k,p}>0$.
These estimates are consequences of Meyer inequalities (see %
\cite[Proposition 1.5.7]{nualart2006malliavin}).
 In particular, these estimates imply that
$\mathbb{D}^{q,2}(\EuFrak H^{\otimes q})\subset \mathrm{Dom}\delta ^{q}$ for any integer $q\geq 1$.

The following commutation relationship between the
Malliavin derivative and the Skorohod integral (see \cite[Proposition 1.3.2]{nualart2006malliavin}) is also useful:
\begin{equation}
D\delta (u)=u+\delta (Du),  \label{comm1}
\end{equation}%
for any $u\in \mathbb{D}^{2,2}(\EuFrak H)$. By induction we can show
the following formula for any symmetric element $u$ in $\mathbb{D}^{j+k,2}(%
\EuFrak H^{\otimes j})$
\begin{equation}
D^{k}\delta ^{j}(u)=\sum_{i=0}^{j\wedge k}\binom{k}{i}\binom{j}{i}i!\delta
^{j-i}(D^{k-i}u).  \label{t5}
\end{equation}%
In particular, when $j=k$,  making the substitution $i\rightarrow k-i$, we obtain
\begin{equation} 
\label{t6} D^k \delta^k (u) = \sum_{i=0}^k \binom ki^2 (k-i)! \delta^i (D^i u).
\end{equation}

We will also use the following formula for the multiple Skorohod integral.
\begin{lemme}\label{biglem}
If $\varphi$ is a $q$ times continuously differentiable function on $\mathbb{R}$ such that $\varphi^{(q)}$ has moderate growth and  $g,h\in \HH$, then
\[
\delta^q(\varphi(B(g)) h^{\otimes q})=\sum_{r=0}^q \binom q r \varphi^{(r)} (B(g)) \langle h,g\rangle_{\mathfrak{H}}^r I_{q-r}(h^{\otimes q-r}) (-1)^r.
\]
\end{lemme}
\begin{proof}
We will prove this formula by induction. By linearity, it suffices to assume that $\|h\|_\HH=1$.
When $q=1$, this formula reduces to 
\[
\delta(\varphi(B(g))h) = \varphi(B(g))\delta(h)- \varphi'(B(g)) \langle h,g\rangle_{\mathfrak{H}},
\]
which is a particular case of   (\ref{t3}) with $q=1$.
Suppose it holds for $q$. 
Using the recurrence formula $H_{n+1}(x) = xH_n(x)-nH_{n-1}(x)$ for Hermite polynomials and the relation between Hermite polynomials and multiple stochastic integrals, we can write
\begin{align}
\begin{split}\label{eq1}
I_{q}(h^{\otimes q}) &=  H_q(\delta(h) )\\
&=   \delta(h) H_{q-1}(\delta(h) ) - (q-1)  H_{q-2}(\delta(h) )\\
&=\delta(h) I_{q-1}(h^{ q-1})- (q-1)  I_{q-2}(h^{\otimes q-2}).
\end{split}
\end{align}
Applying the inductive hypothesis, we have
\begin{align}  \notag
\delta^{q+1}(\varphi(B(g))h^{\otimes q+1}) &=\delta(\delta^q(\varphi(B(g)) h^{\otimes q})h)\\  \notag
&= \sum_{r=0}^q \binom qr \delta\left( \varphi^{(r)} (B(g)) \langle h,g\rangle_{\mathfrak{H}}^r I_{q-r}(h^{\otimes q-r})  h\right)(-1)^r\\ \notag
&= \sum_{r=0}^q \binom qr (-1)^r [\delta(h)I_{q-r}(h^{\otimes q-r}) \varphi^{(r)}(B(g))\langle h,g\rangle_{\mathfrak{H}}^r\\  \label{qq1}
&\msplitsp-\langle D[\varphi^{(r)}(B(g))\langle h,g\rangle^r \delta^{q-r}(h^{\otimes q-r})],h\rangle_{\mathfrak{H}} ].
\end{align}
Computing  the   derivative in the last term, we have
\begin{align}
\begin{split}\label{eq2}
&\msplitsp\langle D[\varphi^{(r)}(B(g))\langle h,g\rangle^r I_{q-r}(h^{\otimes q-r})],h\rangle_{\mathfrak{H}}\\
&= \varphi^{(r+1)} (B(g)) \langle h,g\rangle_{\mathfrak{H}}^{r+1} I_{q-r}(h^{\otimes q-r})\\
&+\varphi^{(r)} (B(g)) \langle h,g\rangle_{\mathfrak{H}}^{r}(q-r) I_{q-r-1} (h^{\otimes q-r-1}),
\end{split}
\end{align}
Substituting  \eqref{eq2} into  (\ref{qq1}) and using (\ref{t3}), yields
\begin{align*}
\delta^{q+1}(\phi(B(g))h^{\otimes q+1}) &= \sum_{r=0}^q \binom qr (-1)^r [\delta(h)I_{q-r}(h^{\otimes q-r}) \varphi^{(r)}(B(g))\langle h,g\rangle_{\mathfrak{H}}^r\\
&-\varphi^{(r+1)} (B(g)) \langle h,g\rangle_{\mathfrak{H}}^{r+1} I_{q-r}(h^{\otimes q-r})\\
&-\varphi^{(r)} (B(g)) \langle h,g\rangle_{\mathfrak{H}}^{r}(q-r) I_{q-r-1} (h^{\otimes q-r-1})  ]\\
&= \sum_{r=0}^q \binom qr (-1)^r\varphi^{(r)}(B(g))\langle h,g\rangle_{\mathfrak{H}}^r I_{q-r+1}(h^{\otimes q-r+1}) \\
&+\sum_{r=1}^{q+1} \binom q{r-1} (-1)^r \varphi^{(r)} (B(g)) \langle h,g\rangle_{\mathfrak{H}}^{r} I_{q-r+1}(h^{\otimes q-r+1}). 
\end{align*}
Taking into account that $\binom q r + \binom q{r-1} = \binom {q+1}r$, we finally obtain
\begin{align*}
\delta^{q+1}(\phi(B(g))h^{\otimes q+1}) &= \sum_{r=0}^{q+1} \binom{q+1}r (-1)^r \varphi^{(r)}(B(g)) \langle h,g\rangle_{\mathfrak{H}}^r I_{(q+1)-r}(h^{\otimes ((q+1)-r)})
\end{align*}
and the induction is complete.
\end{proof}

\subsection{Rate of convergence to a mixture of Gaussian laws}

In this subsection we state Theorem 5.2 from \cite{nourdin2016} here, that will play a basic role in the proof of our main theorem.

\begin{prop} \label{p1}
Suppose that $u\in \mathbb{D}^{2q,4q}(\mathfrak{H}^{q})$ is symmetric. Let $%
F=\delta ^{q}(u)$. Let $S\in \mathbb{D}^{q,4q}$, and let $\eta \sim N(0,1)$
indicate a standard Gaussian random variable, independent of the fractional Brownian motion $B$.  Assume that $\varphi :\mathbb{R}\rightarrow \mathbb{R}$
is $C^{2q+1}$ with $\Vert \varphi ^{(k)}\Vert _{\infty }<\infty $ for any $%
k=0,\ldots ,2q+1$. Then
\begin{eqnarray*}
 && \big|E[\varphi (F)]-E[\varphi (S\eta )]\big|
    \leq \frac{1}{2}\Vert \varphi ^{\prime \prime }\Vert _{\infty }\,E%
\big[|{\langle u,D^qF\rangle _{\EuFrak H^{\otimes q}}}-S^{2}|\big] \\
&&\quad +\sum_{(b,b')\in \mathcal{Q}}\ \sum_{j=0}^{ {\lfloor \left| b' \right| /2 \rfloor  %
} } c_{q,b,b',j}  \left\| \varphi ^{(1+\left| b\right| +2\left| b'\right|
-2j)}\right\| _{\infty }    \\
&&\quad \times E\left[S^{ |b'|-2j}
 \left| \left\langle u,\left( DF\right) ^{\otimes
b_{1}}\otimes \cdots \otimes \left( D^{q-1}F\right) ^{\otimes b_{q-1}}\otimes
\ \left( DS\right) ^{\otimes b'_{1}}\otimes \cdots \otimes \left(
D^{q}S\right) ^{\otimes b'_{q}}\right\rangle _{\mathfrak{H}^{\otimes q}} \right|
\right],
\end{eqnarray*}%
where $\mathcal{Q}$ is the set of all pairs of vectors $b=(b_{1},b_{2},%
\ldots b_{q-1})$ and $b'=(b'_{1},\ldots ,b'_{q})$ of nonnegative integers
satisfying the constraints $b_{1}+2b_{2}+\cdots +(q-1)b_{q-1}+b'_{1}+2b'_{2}+\cdots
qb'_{q}=q$. { The constants $c_{q,b,b',j}$ are given  by
\[
c_{q,b,b',j} =\frac 12 B(|b|/2 +1/2, |b'|/2+1) \prod_{i=1}^q  \binom{c_i}{b_i}  \times
\frac { |b'|!}{ 2^j (|b'| -2j)! j!}   \times \frac{ q!}{ \prod_{i=1}^q i!^{c_i} c_i!},
\]
where $c=b+b'$ and $B$ denotes the Beta function.}
\end{prop}

\section{Proof of Theorem \ref{mainresult}}

Along the proof $C$ will denote a generic  constant that might depend on $q$ and $H$.
Before starting the proof let us make some remarks on the exponent $\phi(H)$ defined in \eqref{eqphi}. Notice that  $H < 1/(2q-2)$ if and only if  $-H< -qH + \frac 12$ and  $H > 1-1/(2q-2)$ if and only if $1-H< q(H-1) + \frac 12$. As a consequence, we have
\[
\phi(H) = \begin{cases}
-qH+\frac12 & 1/2q < H \le 1/(2q-2),\\
-H & 1/(2q-2)  < H \le 1/2,\\
H-1 & 1/2 < H \le 1-1/(2q-2),\\
q(H-1)+\frac12  & 1-1/(2q-2) \le  H < 1- 1/2q.
\end{cases}
\]
 This implies that 
 \begin{equation}
 \phi(H) = \max\left\{-H,H-1,-qH+\frac12,q(H-1)+\frac12  \right\}.\label{phih}
 \end{equation}

The proof will be done in several steps.
  Consider the element $u \in \mathbb{D}^{2q,\infty} := \cap_{p\ge 1} \mathbb{D}^{2q,p}$ given by
\begin{align*}
u_n &= n^{qH-1/2} \sum_{k=0}^{n-1} f(B_{k/n}) \delta_{k/n}^{\otimes q},
\end{align*}
where we recall that  $\delta_{k/n}=\1_{[k/n,(k+1)/n]}$.
 Note first that the random variable $F_n$ does not coincide with $\delta^q(u_n)$, except in the case $H=1/2$. For this reason, we define $G_n=\delta^q(u_n)$  and  first estimate the difference $F_n-G_n$.

 \medskip
 \noindent
 {\it Step1.} \quad We claim that 
\begin{equation}\label{eq7}
E[|F_n-G_n|]\le C\|f\|_{2q-1,2} n^{\phi(H)}.
\end{equation}

To show \eqref{eq7}, we apply Lemma \ref{biglem} and obtain 
\begin{align*}
\delta^q(u_n)  &=n^{qH-1/2} \sum_{k=0}^ {n-1} \delta^q(f(B_{k/n}) \delta_{k/n}^{\otimes q})\\
&=n^{qH-1/2} \sum_{k=0}^{n-1} \sum_{r=0}^q (-1)^r\binom qr f^{(r)}(B_{k/n}) \alpha_{k,k/n}^r I_{q-r}(\delta_{k/n}^{\otimes q-r}).
\end{align*}
Note that the $r=0$ term corresponds to $F_n$, so
\begin{align*}
F_n-G_n&= n^{qH-1/2} \sum_{k=0}^{n-1} \sum_{r=1}^q (-1)^{r+1}\binom qr f^{(r)}(B_{k/n}) \alpha_{k,k/n}^r I_{q-r}(\delta_{k/n}^{\otimes q-r}).
\end{align*}
Let 
\[ K_{n,r} = n^{qH-1/2} \sum_{k=0}^{n-1}(-1)^{r+1} \binom qr f^{(r)}(B_{k/n}) \alpha_{k,k/n}^r I_{q-r}(\delta_{k/n}^{\otimes q-r}) \]
so that $F_n-G_n=\sum_{r=1}^q K_{n,r}$.

Applying the product formula of multiple stochastic integrals (\ref{prodform}), we can write 
\begin{align}  \notag
E[K_{n,r}^2] &=n^{2qH-1}\sum_{j,k=0}^{n-1} \binom qr^2 E\left[f^{(r)}(B_{j,n}) f^{(r)}(B_{k/n})I_{q-r}(\delta_{j/n}^{\otimes q-r}) I_{q-r}(\delta_{k/n}^{\otimes q-r}) \right]\alpha_{j,j/n}^r \alpha_{k,k/n}^r\\  \notag
&= n^{2qH-1}\sum_{j,k=0}^{n-1} \sum_{s=0}^{q-r}\binom qr^2  \binom {q-r} s ^2 s!  \\
&\quad \times   E\left[ f^{(r)}(B_{j,n}) f^{(r)}(B_{k/n})I_{2q-2r-2s} ( \delta_{j/n}^{\otimes q-r-s}\widetilde \otimes \delta_{k/n}^{\otimes q-r-s})   \right]  \beta_{j,k}^s  \alpha_{j,j/n}^r \alpha_{k,k/n}^r.  \label{ecu1}
\end{align}
Using the duality formula between multiple stochastic integrals and the derivative operator \eqref{dual}, we obtain
\begin{align*} 
 &E\left[ f^{(r)}(B_{j,n}) f^{(r)}(B_{k/n})I_{2q-2r-2s} ( \delta_{j/n}^{\otimes q-r-s}\widetilde \otimes \delta_{k/n}^{\otimes q-r-s})   \right]\\ 
 & \qquad = E\left[ \langle D^{2q-2r-2s} (f^{(r)}(B_{j,n}) f^{(r)}(B_{k/n})), \delta_{j/n}^{\otimes q-r-s} \widetilde \otimes \delta_{k/n}^{\otimes q-r-s}
 \rangle_{\HH^{\otimes 2q-2r-2s}}\right].
 \end{align*}
Finally, applying the Leibniz rule,
 we can write
\begin{align}  \notag
 &E\left[ f^{(r)}(B_{j,n}) f^{(r)}(B_{k/n})I_{2q-2r-2s} ( \delta_{j/n}^{\otimes q-r-s}\widetilde \otimes \delta_{k/n}^{\otimes q-r-s})   \right]\\ \notag 
& \qquad = \sum_{m=0}^{2q-2r-2s} \binom{2q-2r-2s}{m} E\left[  f^{(r+m)}(B_{j,n}) f^{(2q-r-2s-m)}(B_{k,n}) \right] \\
& \qquad \qquad  \times \langle \1_ {[0,j/n]} ^{\otimes m} \otimes \1_{[0,k/n]} ^{\otimes (2q-r-2s-m)}, \delta_{j/n}^{\otimes q-r-s}\widetilde \otimes \delta_{k/n}^{\otimes q-r-s} \rangle_{\HH^{\otimes 2q-2r-2s}} . \label{ecu2}
 \end{align}
Substituting (\ref{ecu2}) into  (\ref{ecu1}), yields
 \begin{align*}
E[K_{n,r}^2] &\le C \|f\|_{2q-1,2}^2 n^{2qH-1} \sum_{j,k=0}^{n-1}\sum_{s=0}^{q-r}  |\beta_{j,k}|^s |\alpha_{j,j/n}|^r  |\alpha_{k,k/n}|^r\\
& \qquad  \times
\sum_{m=0}^{2q-2r-2s} \sum_{\mathclap{\substack{0\le i\le m\\i\vee (m-i) \le q-r-s}}} |\alpha_{j,j/n}|^i  |\alpha_{j,k/n}|^{q-r-s-i}  |\alpha_{k,j/n}|^{m-i}  |\alpha_{k,k/n}|^{q-r-s-m+i}\\
&=C \|f\|_{2q-1,2}^2 n^{2qH-1} \sum_{j,k=0}^{n-1}\sum_{s=0}^{q-r}  |\beta_{j,k}|^s  \\
& \qquad \times 
\sum_{m=0}^{2q-2r-2s} \sum_{\mathclap{\substack{0\le i\le m\\i\vee (m-i) \le q-r-s}}} |\alpha_{j,j/n}|^{i+r} |\alpha_{j,k/n}|^{q-r-s-i} |\alpha_{k,j/n}|^{m-i} | \alpha_{k,k/n}|^{q-s-m+i}.
\end{align*}
 Then, decomposing the summation in $s$ into the cases  $s=0$ and $s\ge 1$, we obtain
\begin{align*}
E[K_{n,r}^2]&\le C \|f\|_{2q-1,2}^2 n^{2qH-1}  \\
& \times \left(\sum_{j,k=0}^{n-1} \sum_{m=0}^{2q-2r}  \hspace{0.7em}\sum_{\mathclap{\substack{0\le i\le m\\i\vee (m-i) \le q-r}}} |\alpha_{j,j/n}|^{i+r} |\alpha_{j,k/n}|^{q-r-i} |\alpha_{k,j/n}|^{m-i} |\alpha_{k,k/n}|^{q-m+i}  \right.\\
&\left.+\sum_{j,k=0}^{n-1}\sum_{s=1}^{q-r}  |\beta_{j,k}|^s 
\sum_{m=0}^{2q-2r-2s}\sum_{\mathclap{\substack{0\le i\le m\\i\vee (m-i) \le q-r-s}}} |\alpha_{j,j/n}|^{i+r} |\alpha_{j,k/n}|^{q-r-s-i} |\alpha_{k,j/n}|^{m-i} |\alpha_{k,k/n}|^{q-s-m+i}\right).\\
\end{align*}
In the $s=0$ case, we replace the summation of $j$ and $k$ with a factor of $n^2$ and estimate the $\alpha$'s with Lemma \ref{biglem}(a). For $s\ge 1$, we apply Lemma \ref{betas}(b) and bound each $\alpha$ with Lemma \ref{biglem}(a), so that
\begin{align*}
E[K_{n,r}^2]&\le C \|f\|_{2q-1,2}^2 n^{2qH-1} \left( n^{-2q(2H\wedge 1)+2}
+\sum_{s=1}^{q-r}  n^{(1-2sH)\vee(2-2s)} n^{-(2H\wedge 1)(2q-2s)}\right).
\end{align*}
 The $s=0$ term yields the contribution $n^{2qH-1} n^{-2q(2H\wedge 1)+2} =n^{-2q(H\wedge (1-H))+1}$. Note that $-(H\wedge(1-H))q+\frac12\le \phi(H)$. Let us consider the terms 
 \[
 A_s := C \|f\|_{2q-1,2}^2 n^{2qH-1}  n^{(1-2sH)\vee(2-2s)} n^{-(2H\wedge 1)(2q-2s)},
 \]
 $s=1,2,\ldots, q-r$. We consider three different cases:

\noindent\textit{Case 1.} Suppose that $1-2sH<2-2s$.
In this case, $H>1-1/(2s)\ge1/2$ and $s<1/(2(1-H))$. Therefore, we obtain
\begin{align*}
A_s&= C \|f\|_{2q-1,2}^2 n^{2qH-1} n^{2-2q}=C  \|f\|_{2q-1,2}^2 n^{1+2q(H-1)}.
\end{align*}

\noindent\textit{Case 2.} Suppose that $1-2sH\ge2-2s$ and $H\ge 1/2$.
In this case, $1/2\le H\le 1-1/(2s)\le 1-1/(2(q-1))$ and $ 1/(2(H-1))\le s\le q-1$. So, we can write
\begin{align*}
A_s&= C \|f\|_{2q-1,2}^2 n^{2qH-1}  n^{1-2q+2s(1-H)}\\
&\le C \|f\|_{2q-1,2}^2 n^{2qH-1}  n^{1-2q+2(q-1)(1-H)}\\
&\le C \|f\|_{2q-1,2}^2 n^{2H-2}.
\end{align*}

\noindent\textit{Case 3.} Suppose that $1-2sH\ge2-2s$ and $H< 1/2$.
We have
\begin{align*}
A_s&= C \|f\|_{2q-1,2}^2 n^{2qH-1}   n^{2sH-4qH+1}\\
&\le C  \|f\|_{2q-1,2}^2 n^{2qH-1}n^{2(q-1)H-4qH+1}\\
&= C  \|f\|_{2q-1,2}^2 n^{-2H}.
\end{align*}

Combining the bound for $s=0$ and the bounds for $A_s$,  applying Cauchy-Schwarz's inequality  and taking into account the definition of $\phi(H)$ and \eqref{phih}, 
we obtain
\begin{equation}
E[|F_n-G_n|] \le \sum_{r=1}^q E[|K_{n,r}|] \le \sum_{r=1}^q \sqrt{E[K_{n,r}^2]} \le C\|f\|_{2q-1,2} n^{\phi(H)},\label{fngn}
\end{equation}
thus proving \eqref{eq7}. 

\medskip
\noindent
{\it Step 2.} \quad 
Now that we have a bound for $E[|F_n-G_n|]$, we will establish a bound for  \break $\left|E[\varphi(G_n)]-E[\varphi(S\eta)]\right|$ using Proposition \ref{p1}. First, however, we need a result to convert derivatives of $S$ into derivatives of $S^2$.
By the Faa di Bruno formula (see \cite{Mis} Theorem 2.1), with $h(x)=\sqrt{x}$, we have
\begin{align}\label{eq36}
D^k S = D^k \sqrt{S^2} &= \sum_{m_1+2m_2 + \cdots + km_k=n}C_{m_1, m_2, \ldots, m_k} h^{(m_1 + \cdots + m_k)} (S^2) \otimes_{j=1}^k (D^jS^2)^{\otimes m_j}\nonumber\\
&=  \sum_{m_1+2m_2 + \cdots + km_k=k}C'_{m_1, m_2, \ldots, m_k}S^{1- 2(m_1+\cdots + m_k)}\otimes _{j=1}^k (D^jS^2)^{\otimes m_j},
\end{align}
where the $m_j$ represent the powers of $D^j S^2$, $C_{m_1, \ldots, m_k}=\frac{k!}{m_1!1!^{m1} \ldots m_k! k!^{m_k}}$ is a combinatorial constant that depends on $m_1, \ldots, m_k$ and
\[
C'_{m_1, \ldots, m_k} =C_{m_1, \ldots, m_k} \prod_{\ell=0} ^{ m_1 + \cdots +m_k -1} \left( \frac 12- \ell\right).
\]
Applying this result to each derivative of $S$ in Proposition \ref{p1} and combining the terms, we obtain
\begin{eqnarray}  \notag
 && \left| E[\varphi(G_n)]-E[\varphi(S\eta)] \right|  \le  \frac12 \| \varphi''\|_\infty  E\left[\left| \langle u,D^q G_n\rangle_{\mathfrak{H}^{\otimes q}} -S^2\rangle \right|\right] \\ \notag 
&& \qquad  +C \sum_{(b,b')\in \mathcal{Q} }\sum_{j=0}^{\floor{|b'|/2}}   \|\varphi^{(1+|b|+2|b'| -2j)}\|_{\infty}\\  \notag
&&  \qquad \times \sum_{m_1,m_2, \ldots, m_q} E\Bigg[  S^{2|b'|-2j-2\sum_{k=1}^q |m_k|b'_k}\\
  &&\qquad  \times \left|\left\langle u_n,\bigotimes_{\ell=1}^{q-1} (D^{\ell}G_n)^{\otimes b_{\ell}}
  \times \otimes\bigotimes_{\ell'=1}^q (D^{\ell'}S^2)^{\otimes (b_1' m_{1\ell'} + \cdots +b_q' m_{q\ell'})}\right\rangle_{\mathfrak{H}^{\otimes q}}\right|\Bigg],
\label{beforejacd}
\end{eqnarray}
where  $\mathcal{Q}$ is the set of all vectors $b=(b_1, \dots, b_{q-1})$ and $b'=(b'_1, \dots, b' _q)$ of nonnegative integers such that $b_1+2b_2+\cdots + (q-1)b_{q-1}+b_1'+2b_2'+\cdots + qb_q' = q$. Also, we will use the notation $m_k =(m_{kj})_{j=1,\dots, q}$, $|m_k| = m_{k1} + \cdots +m_{kq}$,  
where the $m_k$ satisfy 
  \begin{equation} \label{qq2}
  m_{i1} + 2m_{i2} + \cdots + qm_{iq}=i.
  \end{equation}
  for each $i=1,\dots, q$.
   We include the combinatorial coefficient from the Faa di Bruno formula in the constant $C$.

Let $d_{\ell'} =b_1'm_{1\ell'}+\cdots b_q' m_{q\ell'}$. Using (\ref{qq2}), we obtain
\[
d_1+2d_2+\cdots + qd_q=\sum_{\ell'=1}^q \ell'(b_1'm_{1\ell'}+\cdots  +b_q' m_{q\ell'}) = b_1'+2b_2'+\cdots +qb_q'.
\]
Therefore,  
\begin{equation}\label{bdsum}b_1+2b_2+\cdots +(q-1)b_{q-1}+d_1+2d_2+\cdots + qd_q = q.\end{equation}
Note that
\[
|d| = b_1'|m_1| + \cdots + b_q' |m_q|.
\]

The exponent of $S$ in \eqref{eq36} is a negative number, denoted by $a$, such that, if $|b'|$ is even,
\begin{align*}
a &=2 |b'| -2j -2\sum_{k=1}^q  |m_k| b'_k  \\
&\ge  2|b'| -2\floor{|b'|/2}  -2\sum_{k=1}^q  |m_k| b'_k  \\
&\ge |b'|-2\sum_{k=1}^q  |m_k| b'_k.
  \end{align*}
  Noting that $|m_k|\le k$, this implies
\[
 a\ge \sum_{k=1}^q (1-2k)b_k'.
\]
Similarly, if $|b'|$ is odd,
\[a\ge 1+ \sum_{k=1}^q (1-2k) b'_k.\]
The lowest possible {value }of $a$ is obtained when $m_{q1}=q$, $b'_q=1$, $b_1'=b_2'=\cdots = b_{q-1}'=0$, so
\[
a \ge 2-2q.
\]
 For any $a\in \{0, -1,, \ldots, 2-2q\} $, define
\begin{align}\begin{split}\label{jacd}
J_{a,b,d}&:= E\left[ S^{a}  \left|\left\langle u_n,\bigotimes_{\ell=1}^{q-1} (D^{\ell}G_n)^{\otimes b_{\ell}}\otimes\bigotimes_{\ell'=1}^q (D^{\ell'}S^2)^{\otimes d_{\ell'}}\right\rangle_{\mathfrak{H}^{\otimes q}}\right|\right]\\
&=n^{qH-1/2} E\left[ S^{a}  \left|   \sum_{j=0}^{n-1}f(B_{j/n}) \prod_{\ell=1}^{q-1}\langle \delta_{j/n}^{\otimes \ell},D^\ell G_n\rangle_{\mathfrak{H}^{\otimes \ell}}^{b_\ell} \prod_{\ell'=1}^{q}\langle \delta_{j/n}^{\otimes \ell'},D^{\ell'} S^2\rangle_{\mathfrak{H}^{\otimes \ell'}}^{d_{\ell'}}  \right|\right],
\end{split}
\end{align}
where 
\begin{equation}\label{bd}
b_1+2b_2+\cdots +(q-1)b_{q-1}+d_1+2d_2+\cdots + qd_q = q.
\end{equation}
Notice also that 
\[1+|b|+2|b'|-2j\le 1+|b|+2|b'| \le  2q+1.\]
Then, from \eqref{beforejacd} we conclude that
\begin{eqnarray}
\left| E[\varphi(G_n)]-E[\varphi(S\eta)] \right| &\le  & \frac12 \| \varphi''\|_\infty  E\left[\left| \langle u,D^q G_n\rangle_{\mathfrak{H}^{\otimes q}} -S^2\rangle \right|\right] \nonumber\\
 & & +C \sup_{1\le i\le 2q+1}\|\varphi^{(i)}\|_{\infty}\sup_{(b,d)\in \mathcal{Q}}\,  \sup _{ 2-2q \le a \le 0} J_{a, b,d}.\label{t35}
\end{eqnarray}

\medskip
\noindent
{\it Step 3.} \quad 
We next show that 
\begin{equation}\label{eqcow}
E\left(\left| \left\langle u_n, D^q G_n\right\rangle -S^2  \right| \right)\le C \|f\|_{2q,2}^2 n^{\phi(H)}.
\end{equation}
Recall that $G_n=\delta^n(u_n)$.
We have, applying \eqref{t6},
\begin{align}\begin{split}\label{eqbigest}
E \left( \left| \left\langle u_n, D^q G_n\right\rangle -S^2  \right|  \right) &\le E \left( q!\left | \|u_n\|^2_{\mathfrak{H}^{\otimes q}}-S^2\right| \right)+\sum_{i=1}^q \binom qi^2 (q-i)! E \left( q!\left|\langle u_n,\delta^{i}(D^{i} u_n)\rangle_{\mathfrak{H}^{\otimes q}}\right| \right).
\end{split}\end{align}
Let $A_n :=  \left |q! \|u_n\|^2_{\mathfrak{H}^{\otimes q}}-S^2\right|$ and $B_{n,i} :=|\langle u_n,\delta^{i}(D^{i} u_n)\rangle_{\mathfrak{H}^{\otimes q}}|$, so that we can write 
\[ 
E \left( \left| \left\langle u_n, D^q G_n\right\rangle -S^2  \right| \right) \le E[A_n] + C\sum_{i=1}^q E( B_{n,i}).
\]

First, we will show that
\begin{align*}
E[A_n]=E\left[q!\left| \|u_n\|^2_{\mathfrak{H}^{\otimes q}}-S^2\right|\right]\le  C\|f\|_{1,2}^2 n^{\phi(H)}.
\end{align*}
We have
\begin{align*}
q!\|u_n\|_{\mathfrak{H}^{\otimes q}}^2 &= q!n^{2qH-1} \sum_{j,j'=0}^{n-1} f(B_{j/n})f( B_{j'/n}) \beta_{j,j'}^q\\
&= q!n^{-1} \sum_{j,j'=0}^{n-1} f(B_{j/n})f( B_{j'/n})\rho_H(k-j)^q\\
&= q!\frac1n \sum_{p=-n+1}^{n-1} \sum_{j=0\vee -p}^{(n-1)\wedge (n-1-p)} f(B_{j/n}) f(B_{(j+p)/n}) \rho_H(p)^q\\
:&= P_n+Q_n,
\end{align*}
where \[P_n:= \frac{q!}n \sum_{p=-n+1}^{n-1} \sum_{j=0\vee -p}^{(n-1)\wedge (n-1-p)} f(B_{j/n}) (f(B_{(j+p)/n})-f(B_{j/n})) \rho_H(p)^q\] \\and \[Q_n:=\frac{q!}n \sum_{p=-n+1}^{n-1} \sum_{j=0\vee -p}^{(n-1)\wedge (n-1-p)} f(B_{j/n})^2\rho_H(p)^q.\] 
Using $E[(f(B_{(j+p)/n})-f(B_{j/n}))^2]^{1/2}\le C\|f\|_{1,2} n^{-H}$ and the fact that $\sum_{p=-\infty}^\infty |\rho_H(p)|^q<\infty$, we have
\begin{equation}
E[|P_n|] \le C\|f\|_{1,2}^2n^{-H}.\label{pn}
\end{equation}
Next, taking into account that $\sum_{|p|\ge n} |\rho_H(p)|^q$ converges to zero at the rate $n^{q(2H-2)+1}$ as $n\to\infty$, we can write 
\begin{eqnarray*}
E\left[\left| Q_n-S^2\right|\right]
&\le& C\|f\|_{0,2}^2 n^{q(2H-2)+1}\\
&& +q!\sum_{k=-\infty}^\infty \rho_H(k)^qE\left[\left| \frac1n  \sum_{j=0}^{n-1} f^2(B_{j/n})-\int_0^1 f^2(B_s)\,ds\right| \right]\nonumber\\
&=&  C\|f\|_{0,2}^2 n^{q(2H-2)+1}\\
&&+ q!\sum_{k=-\infty}^\infty \rho_H(k)^qE\left[\left|  \sum_{j=0}^{n-1}\int_{j/n}^{(j+1)/n} f^2(B_{j/n}) -f^2(B_s)\,ds\right| \right]
\end{eqnarray*}
Using $E[|f^2(B_{j/n})-f^2(B_s)|]\le C\|f\|_{1,2}^2 n^{-H}$ for $s\in [j/n,(j+1)/n]$, we obtain
\begin{equation}
E[|Q_n-S^2|] \le C\|f\|_{0,2}^2 n^{q(2H-2)+1} + C\|f\|_{1,2}^2 n^{-H}.\label{qn}
\end{equation}
Because $E[A_n]\le E[|P_n|] + E[|Q_n-S^2|]$, we have from \eqref{pn} and \eqref{qn} that
\begin{align*}
\begin{split}
E\left[A_n\right] 
&\le C(\|f\|_{1,2}^2 n^{-H} + \|f\|_{0,2}^2 n^{q(2H-2)+1})\\
&\le C\|f\|_{1,2}^2 n^{\phi(H)}.
\end{split}\end{align*}

Next, we estimate the terms $E[B_{n,i}]$ for $i=1,\dots, q$. Taking into account  the definition of $u_n$, we obtain
\begin{align*}
 E[B_{n,i}] &=E\left[\left| \left\langle u_n, \delta^i(D^i u_n)\right\rangle_{\mathfrak H^{\otimes q}} \right | \right ]\\
&\le n^{qH-1/2} \sum_{j=0}^{n-1} E\left[\left| f(B_{j/n}) \langle \delta_{j/n}^{\otimes q},\delta^i(D^i u_n)\rangle_{\mathfrak{H}^{\otimes q}}  \right | \right ]\\
&= n^{qH-1/2} \sum_{j=0}^{n-1} E\left[\left| f(B_{j/n}) \delta^i ( \underbrace{D_{j/n} \cdots D_{j/n}}_\text{$i$ times} (u_n \otimes_{q-i} \delta_{j/n}^{\otimes q-i})) \right | \right],
\end{align*}
where here we made use of the  the notation 
\begin{equation} \label{ec1}
D_{j/n}F=\langle DF,\delta_{j/n}\rangle_{\mathfrak{H}}.
\end{equation}
Applying H\"older's and Meyer's inequalities \eqref{Me2}, we have
\begin{align*}
 E[B_{n,i}] &\le  n^{qH-1/2} \|f\|_{0,2}\sum_{j=0}^{n-1} E\left[\left|  \delta^i ( \underbrace{D_{j/n} \cdots D_{j/n}}_\text{$i$ times} (u_n \otimes_{q-i} \delta_{j/n}^{\otimes q-i})) \right |^2 \right]^{1/2}\\
 &\le Cn^{qH-1/2} \|f\|_{0,2} \sum_{j=0}^{n-1}\left\|   \underbrace{D_{j/n} \cdots D_{j/n}}_\text{$i$ times} (u_n \otimes_{q-i} \delta_{j/n}^{\otimes q-i}) \right\|_{i,2} .
\end{align*}

We consider several cases and apply Lemma \ref{a2} with $M=i$, $a=i$, $b=q-i$, $c=q-i$, and $p=2$ to control the Sobolev norm $\|\cdot \|_{i,2}$.

\noindent\medskip \emph{Case 1.}  Suppose that $H\le1/2$, $i<q$.
 We have
\begin{align*}
 E[B_{n,i}] &\le  Cn^{qH-1/2} \|f\|_{2i,2}^2 \sum_{j=0}^{n-1} n^{-1/2-H(q+i)} \\
&=Cn^{qH-1/2} \|f\|_{2i,2}^2n^{1/2-H(q+i)} \\
 &\le C\|f\|_{2i,2}^2 n^{-H}\\
 &\le C\|f\|_{2q,2}^2 n^{\phi(H)}.
 \end{align*}


 \noindent\medskip \emph{Case 2.} Suppose that $H>1/2, i<q$.
We have
 \begin{align*}
 E[B_{n,i}] &\le Cn^{qH-1/2}\|f\|_{2i,2}^2 n\cdot n^{-1/2-i + (-H(q-i))\vee (q(H-1)+1-q+i)}\\
 &=C\|f\|_{2i,2}^2  n^{(i(H-1))\vee(2q(H-1)+1)}\\
 &\le C\|f\|_{2i,2}^2  n^{(H-1)\vee(2q(H-1)+1)}\\
 &\le C\|f\|_{2q,2}^2 n^{\phi(H)},
 \end{align*}
 
\noindent  Case 3. Suppose that $i=q$.
Note that $a=q$, $b=0$, and $c=0$. Thus,
 \begin{align*}
 E[B_{n,i}]=E[B_{n,q}] &\le Cn^{qH-1/2}\|f\|_{2q,2}^2 n\cdot n^{-q(2H\wedge 1)}\\
 &=C\|f\|_{2q,2}^2 n^{-q(H\wedge (1-H))+1/2} \\
 &=C\|f\|_{2q,2}^2 n^{(-qH+1/2)\vee (q(1-H)+1/2)} \\
 &\le C\|f\|_{2q,2}^2 n^{\phi(H)}.
 \end{align*}
 This completes the proof of \eqref{eqcow}.

\medskip\noindent{\it Step 4.} \quad 
Next, we will show that $J_{a,b,d}\le C \|f\|_{2q,(2q+2)\beta}^{2q+1}E[S^{(2-2q)\alpha}]^{1/\alpha} n^{\phi(H)}$. Using H\"older's inequality with $1/\alpha+1/\beta=1$, $r:=|b|+|d|+1\le q+1$, we can write
\begin{align*}
 J_{a,b,d} \le Cn^{qH-1/2}\|f\|_{0,r\beta}E[S^{a\alpha}]^{1/\alpha}\sum_{k=0}^{n-1} &\prod_{\ell=1}^{q-1}E\left[\left|\langle \delta_{k/n}^{\otimes \ell},D^\ell G_n\rangle_{\mathfrak{H}^{\otimes \ell}}\right|^{r\beta}\right] ^{b_\ell/r\beta} \\
\times &\prod_{\ell'=1}^q E\left[\left| \langle \delta_{k/n}^{\otimes \ell'},D^{\ell'} S^2\rangle_{\mathfrak{H}^{\otimes \ell'}} \right|^{r\beta}\right]^{d_{\ell'}/r\beta}.
\end{align*}

Let
\[
K_{k,\ell} :=E\left[\left|\langle \delta_{k/n}^{\otimes \ell},D^\ell G_n\rangle_{\mathfrak{H}^{\otimes \ell}}\right|^{r\beta}\right] ^{1/r\beta}
\]
and
\[
L_{k,\ell'} :=E\left[\left| \langle \delta_{k/n}^{\otimes \ell'},D^{\ell'} S^2\rangle_{\mathfrak{H}^{\otimes \ell'}} \right|^{r\beta}\right]^{1/r\beta},
\]
so that
\begin{equation}\label{jabdkl}
 J_{a,b,d} \le Cn^{qH-1/2}\|f\|_{0,r\beta} E[S^{a\alpha}]^{1/\alpha} \sum_{k=0}^{n-1}\prod_{\ell=1}^{q-1}K_{k,\ell}^{b_\ell} \prod_{\ell'=1}^q L_{k,\ell'}^{d_{\ell'}}.
\end{equation}
We will now find estimates for $K_{k,\ell}$ and $L_{k,\ell'}$.
First, applying \eqref{t5} and using the notation (\ref{ec1}), we can write
\begin{align}
\left\langle \delta_{k/n}^{\otimes \ell},D^\ell G_n\right\rangle_{\mathfrak{H}^{\otimes \ell}}  &= 
\sum_{i=0}^\ell \binom \ell i ^2 i! \left\langle \delta_{k/n}^{\otimes \ell}, \delta^{q-i}(D^{\ell-i}u_n)\right\rangle_{\mathfrak{H}^{\otimes \ell}}\nonumber\\
&=\sum_{i=0}^\ell  \binom \ell i ^2 i!  \delta^{q-i}\left(\underbrace{D_{k/n}\cdots D_{k/n}}_{\ell-i \textrm{ times}}(u_n \otimes_i \delta_{k/n}^{\otimes i}) \right)\nonumber,
\end{align}
so by Minkowski's and Meyer's inequalities \eqref{Me2},
\begin{align}\label{kbdl}
K_{k,\ell}&\le C\! \sum_{i=0}^\ell  \left\|   \underbrace{D_{k/n}\cdots D_{k/n}}_{\ell-i \textrm{ times}}(u_n \otimes_i \delta_{k/n}^{\otimes i})  \right\|_{q-i,r\beta}.
\end{align}
Let 
\[M_{k,\ell,i}:= \left\|   \underbrace{D_{k/n}\cdots D_{k/n}}_{\ell-i \textrm{ times}}(u_n \otimes_i \delta_{k/n}^{\otimes i})  \right\|_{q-i,r\beta},
\]
 so that \eqref{kbdl} becomes
$K_{k,\ell}\le C\sum_{i=0}^\ell M_{k,\ell,i}$.
We now apply Lemma \ref{a2} with $M=q-i$, $a=\ell-i$, $b=i$, $c=i$, $p=r\beta$ and consider three cases.

\medskip\noindent \emph{Case 1.} Suppose that $0<i\le \ell, H\le 1/2$.
We have
\begin{align}
M_{k,\ell,i}&\le C\|f\|_{q+\ell-2i,r\beta}n^{-1/2-H(2\ell-i)}\le C\|f\|_{q+\ell-2i,r\beta}n^{-1/2-H(2\ell-\ell)}\nonumber\\
&=C\|f\|_{q+\ell-2(\ell-1),r\beta} n^{-1/2-H\ell}\nonumber\\
&\le C\|f\|_{q+\ell-2(\ell-1),r\beta} n^{(-1/2-H\ell)\vee (-\ell(2H\wedge 1))}\label{m1}
\end{align}

\medskip
\noindent \emph{Case 2.} Suppose that $0<i\le \ell, H> 1/2$.
We have
\begin{align*}
M_{k,\ell,i}&\le C\|f\|_{q+\ell-2i,r\beta}n^{-1/2-\ell+i+(-Hi)\vee (q(H-1)+1-i)}\\
&=C\|f\|_{q+\ell-2i,r\beta}n^{(i(1-H)-\ell-1/2)\vee (q(H-1)+1/2-\ell)}.
\end{align*}
Because 
\[ i(1-H) - \ell -1/2 \le \ell(1-H)-\ell-1/2 = -1/2-H\ell
\]
and
\[
q(H-1)+1/2-\ell \le \phi(H)-\ell\le -\ell,
\]
we have, recalling $H>1/2$,
\begin{equation}\label{m2}
M_{k,\ell,i}\le C\|f\|_{q+\ell-2i,r\beta}n^{\left(-1/2-H\ell\right)\vee (-\ell)}= C\|f\|_{q+\ell-2i,r\beta}n^{\left(-1/2-H\ell\right)\vee (-\ell(2H\wedge 1))}.
\end{equation}

\medskip\noindent \emph{Case 3.} Suppose that $i=0$.
Because $i=0$, we have $M=q$, $a=\ell$, $b=0$, and $c=0$. Thus,
\begin{equation}\label{m3}
M_{k,\ell,i}=M_{k,\ell,0}\le C\|f\|_{q+\ell,r\beta}n^{-\ell(2H\wedge 1)}\le C\|f\|_{q+\ell,r\beta}n^{(-1/2-H\ell)\vee(-\ell(2H\wedge 1))}.
\end{equation}

\medskip
Combining \eqref{m1}, \eqref{m2}, and \eqref{m3}, we have
\begin{equation}\label{prodgs}
K_{k,\ell}\le C\|f\|_{q+\ell,r\beta}n^{\left(-1/2-H\ell\right)\vee (-\ell(2H\wedge 1))}
=C\|f\|_{q+\ell,r\beta}n^{-\ell(2H\wedge 1)+\left(-1/2+\ell(H\wedge(1-H))\right)_+}.
\end{equation}
Let $1\le \ell_0(H)\le q-1$ be chosen such that when $\ell\le \ell_0(H)$, 
\[
-1/2+\ell(H\wedge(1-H))\le 0
\]
and when $ \ell > \ell_0(H)$,
\[
-1/2+\ell(H\wedge(1-H)) > 0.
\]
Observe that 
\[
\ell_0(H) = \min\left\{q-1,\left\lfloor{\frac{1}{2(H\wedge(1-H))}}\right\rfloor\right\}.
\]
When $\ell\le \ell_0(H)$,
\[
K_{k,\ell} \le C\|f\|_{q+\ell,r\beta}n^{-\ell(2H\wedge 1)}
\]
and when $\ell> \ell_0(H)$,
\[
K_{k,\ell} \le C\|f\|_{q+\ell,r\beta}n^{-\ell(2H\wedge 1) + \left(-1/2+\ell(H\wedge(1-H))\right)}.
\]
Thus,
\begin{equation}\label{Ks}
\prod_{\ell=1}^{q-1} K_{k,\ell}^{b_\ell}\le C \|f\|_{2q-1,r\beta}^{|b|} n^{\kappa_1(H)},
\end{equation}
where 
\begin{equation}\label{kappa1}
\kappa_1(H):=-(2H\wedge 1) \sum_{\ell=1}^{q-1} \ell b_\ell+ \!\!\sum_{\ell=\ell_0(H)+1}^{q-1} b_\ell\left(-1/2+\ell(H\wedge(1-H))\right).
\end{equation}

Next, we will estimate $L_{k,\ell'}$.
Let $g(x)=f^2(x)$. Then, applying Lemma \ref{lem1}(a), the semi-norm norm \eqref{norm}, and using Minkowski's inequality, we can write
\begin{align*}
\nonumber L_{k,\ell'}
&= \sigma_{H,q}^2 \left\| \int_0^1 g^{(\ell')}(B_s) \alpha_{k,s}^{\ell'}\,ds \right\|_{\beta r} \\
&\le  Cn^{-\ell' (2H\wedge 1)}  \left\| g\right\|_{\ell',\beta r}.
\end{align*}
Noting that for $k\ge 1$, $g^{(k)}(x) = \sum_{z=1}^k C_z f^{(z)}(x) f^{(k-z)}(x)$,  for some combinatoric numbers $C_z$, we have 
\begin{align}\label{prods}
L_{k,\ell'}&\le C n^{-\ell' (2H\wedge 1)} \|f\|_{\ell',2\beta r}^{2}.
\end{align}
Taking the product over $\ell'$ and applying \eqref{prods}, we obtain
\begin{equation}\label{Ls}
\prod_{\ell'=1}^q L_{k,\ell'}^{d_{\ell'} }\le C \|f\|_{q,2\beta r}^{2|d|} n^{\kappa_2(H)},
\end{equation}
where
\begin{equation}\label{kappa2}
\kappa_2(H) :=-(2H\wedge 1)\sum_{\ell'=1}^q \ell' d_{\ell'}.
\end{equation}


Applying \eqref{Ks} and \eqref{Ls}, we have, recalling \eqref{jabdkl} and replacing the summation in $k$ with the factor $n$,
\begin{align}\begin{split}\label{jabd}
J_{a,b,d}&\le Cn^{qH-1/2} \|f\|_{2q,2r\beta}^{1+|b|+2|d| }E[S^{a\alpha}]^{1/\alpha} n\cdot n^{\kappa_1(H)+\kappa_2(H)}\\
&=C\|f\|_{2q,2r\beta}^{1+|b|+2|d| }E[S^{a\alpha}]^{1/\alpha}  n^{qH+\frac12+\kappa_1(H)+\kappa_2(H)}\\
&=C\|f\|_{2q,2r\beta}^{1+|b|+2|d| }E[S^{a\alpha}]^{1/\alpha}  n^{\kappa(H)},
\end{split}\end{align}
where
\[
\kappa(H):=qH+\frac12+\kappa_1(H)+\kappa_2(H).
\]
Recalling \eqref{kappa1}, \eqref{kappa2}, and \eqref{bd}, we have
\begin{align}
\kappa(H)&= qH+\frac12 -(2H\wedge 1)\left(\sum_{\ell=1}^{q-1}\ell b_\ell + \sum_{\ell'=1}^q \ell' d_{\ell'} \right)+ \!\!\sum_{\ell=\ell_0(H)+1}^{q-1} b_\ell\left(-1/2+\ell(H\wedge(1-H))\right)\nonumber\\
&= -q(H\wedge (1-H))+\frac12+ \!\!\sum_{\ell=\ell_0(H)+1}^{q-1} b_\ell\left(-1/2+\ell(H\wedge(1-H))\right).\label{kappa}
\end{align}
We will now show that
\begin{equation}
\kappa(H)\le \phi(H)\label{kp}.
\end{equation}
We consider three cases depending on the value of $\sum_{\ell=\ell_0(H)+1}^{q-1} b_\ell$. 
%
%

\medskip\noindent\emph{Case 1.} Suppose that $\sum_{\ell=\ell_0(H)+1}^{q-1} b_\ell=0$. Then
\[
\kappa(H) =  -(H\wedge (1-H))q+\frac12= \left(-qH+\frac12\right)\vee\left(q(H-1)+\frac12\right)\le \phi(H).
\]

\medskip\noindent\emph{Case 2.} Suppose that $\sum_{\ell=\ell_0(H)+1}^{q-1} b_\ell=1$. We conclude that all of $b_{\ell_0(H)+1},\cdots, b_{q-1}$ are zero except for one, say $b_m=1$, $\ell_0(H)+1\le m\le q-1$. Because $m\le q-1$, 
we have
\begin{align*}
\kappa(H) &=  -(H\wedge (1-H))q+\frac12+\!\!\sum_{\ell=\ell_0(H)+1}^{q-1}b_\ell\left(-1/2+\ell(H\wedge(1-H))\right)\\
&= -q(H\wedge (1-H))+\frac12(1-b_m)+mb_m(H\wedge(1-H)) \\
&=(m-q)(H\wedge (1-H)) \\
&\le-(H\wedge(1-H))=(-H)\vee(H-1)\le \phi(H).
\end{align*}

\medskip\noindent\emph{Case 3.} Suppose that $\sum_{\ell=\ell_0(H)+1}^{q-1} b_\ell\ge 2$. We have
\begin{align*}
\kappa(H) &=  -q(H\wedge (1-H))+\frac12+\!\!\sum_{\ell=\ell_0(H)+1}^{q-1}b_\ell\left(-1/2+\ell(H\wedge(1-H))\right)\\
&= -q(H\wedge (1-H)) +\frac12\left(1-\sum_{\ell=\ell_0(H)+1}^{q-1}b_\ell \right)+(H\wedge (1-H))\sum_{\ell=\ell_0(H)+1}^{q-1} \ell b_\ell\\
&\le -q(H\wedge (1-H))-\frac12+(H\wedge (1-H)) q=-\frac12\le \phi(H),
\end{align*}
completing the proof of \eqref{kp}. 
Thus, we have, recalling \eqref{jabd},
 \begin{align}
J_{a,b,d}&\le C\|f\|_{2q,2r\beta}^{1+|b|+2|d| }E[S^{a\alpha}]^{1/\alpha}  n^{\phi(H)}\nonumber\\
&\le  C\|f\|_{2q,(2q+2)\beta}^{2q+1}E[S^{(2-2q)\alpha}]^{1/\alpha}n^{\phi(H)}\label{jabd2}.
\end{align}
Combining \eqref{fngn}, \eqref{t35}, \eqref{eqcow}, and \eqref{jabd2}, the proof of Theorem \ref{mainresult} is complete.\qed

 { 
 
 \medskip
 \noindent
 {\bf Remark.} 
In order to  understand the phase transition  in the rate of convergence when $H>1/2$, let us discuss a particular example. Suppose that $q=3$ and $f(x)=x$
and $H>1/2$. Then
\[
F_n=   n^{3H-1/2} \sum_{k=0} ^{n-1} B_{k/n}  I_3(\delta_{k/n}^{\otimes 3}).
\]
The random variable $F_n$ can be decomposed as follows: $F_n=G_n + R_n$, where
\[
G_n= n^{3H-1/2} \sum_{k=0} ^{n-1} I_4 \left( \1_{[0,k/n]} \otimes \delta_{k/n}^{\otimes 3} \right)  
\]
and
\[
R_n=3n^{3H-1/2} \sum_{k=0} ^{n-1} 
\langle \1_{[0,k/n]}, \delta_{k/n} \rangle_{\HH}
I_2(\delta_{k/n}^{\otimes 2}).
\]
Let us find out  the rate of convergence in $L^2$ of the residual term $R_n$.
We can write
\[
E[R_n^2] = \frac 94 n^{-2H-1} \sum_{j,k=1} ^{n-1}   ((k+1)^{2H}- k^{2H} -1) ((j+1)^{2H}- j^{2H} -1)  \rho_H^2(j-k).
\]
Then, if $H\ge \frac 34$, the series   $\sum_{h=1} ^\infty \rho_H^2(h)$ is divergent and the expectation $E[R_n^2] $ behaves,   as
$n^{6H-5}$ when $n\rightarrow \infty$. On the other hand, if $H<\frac 34$, the series $\sum_{h=1} ^\infty \rho_H^2(h)$  is convergent and $E[R_n^2] $ behaves, as
$n^{2H-2}$ when $n\rightarrow \infty$.   We see that the phase transition occurs at $H=\frac 34$.
For a general $q\ge 3$,  and assuming $H>1/2$, the phase transition occurs for values of $q$ and $H$  such that the series
$\sum_{h=1} ^\infty  |\rho_H^{q-1}(h)|$ changes its  convergence, that is, when $(2H-2)(q-1) =-1$. 
One can show that the expectation $E[G_n^2]$ converges to a constant, and the rate  of convergence is worse than that of the residual term.
 }

\section{Appendix}
Here we prove a result we need in the proof of Theorem 1.1. Recall the notation $D_{k/n} F = \langle DF, \delta_{k/n}\rangle_{\mathfrak{H}}$ where $\delta_{k/n} = \1_{[k/n,(k+1)/n]}.$

\begin{lemme}\label{a2}
For any integers $\ell, M, a,b,c$ such that $M\ge 0$, $0\le c\le q$, $0\le a\le q-c$, $0\le b\le c$ and any real number $p>1$,there exists a constant $C$ depending on $q$, $M$, $p$, and the Hurst parameter $H$ such that
\begin{align*}
\left\| \underbrace{D_{k/n}\cdots D_{k/n}}_{\text{$a$ times}} \left(  u_n \otimes_{b} \delta_{k/n}^{\otimes c}  \right) \right\|_{M,p}\le C\|f\|_{M+a,p}n^{\kappa(a,b,c,H)},\end{align*}
where
\[\kappa(a,b,c,H):=\begin{cases}{-1/2-H(2a+c)} & b>0, H\le1/2\\
{-1/2-a-H(c-b)+ (-Hb)\vee (q(H-1)+1-b)} & b>0, H>1/2\\
{-a(2H\wedge 1)-Hc} & b=0
\end{cases}.\]

\end{lemme}
\begin{proof}
For $0\le m\le M$, let
\begin{align*}
A_m&:=E\left[\left\|D^m\left(\underbrace{D_{k/n}\cdots D_{k/n}}_{\text{$a$ times}} \left(  u_n \otimes_{b} \delta_{k/n}^{\otimes c}  \right) \right)\right\|_{\mathfrak{H}^{\otimes (m+q-2b+c)}}^p\right]^{1/p}\\
&=n^{qH-1/2}E\left[\left\|\sum_{j=0}^{n-1} f^{(m+a)}(B_{j/n}) \alpha_{k,j/n}^{a} \beta_{j,k}^b\delta_{k/n}^{\otimes (c-b)} \delta_{j/n}^{\otimes (q-b)} {\otimes} \varepsilon_{j/n}^{\otimes m}\right\|_{\mathfrak{H}^{\otimes (m+q-2b+c)}}^p\right]^{1/p}.
\end{align*}
Taking the norm in $\mathfrak{H}^{m+q-2b+c}$, we have
\begin{align*}
A_m&\le Cn^{qH-1/2} \|f\|_{m+a,p}\!\left\{\sum_{j,j'=0}^{n-1}\!\!|\alpha_{k,j/n}|^a |\alpha_{k,j'/n}|^a |\beta_{j,k}|^b|\beta_{j',k}|^b n^{-2H(c-b)}\! |\beta_{j,j'}|^{q-b}  \right\}^{1/2}.
\end{align*}

We consider three different cases:

\medskip
\noindent \emph{Case 1.} Suppose that $0<b<q$. 
Applying Lemma \ref{threebeta1} to $\sum_{j,j'=0}^{n-1}|\beta_{j,k}|^b |\beta_{j',k}|^b |\beta_{j,j'}|^{q-b}$ and Lemma \ref{lem1}(a) to each of the $\alpha$ terms, we have
\begin{align*}
A_m &\le C n^{qH-1/2} \|f\|_{m+a,p} \left(n^{-2a(2H\wedge 1)}n^{-2H(c-b)} n^{(-2H(q+b))\vee (2-2(q+b))}\right)^{1/2}\\
&=C n^{qH-1/2} \|f\|_{m+a,p} n^{-a(2H\wedge 1)}n^{-H(c-b)} n^{(-H(q+b))\vee (1-(q+b))}.
\end{align*}
When $H\le 1/2$, $q+b\ge 2\ge1/(1-H)$, so $-H(q+b) \ge1-(q+b)$, and we have
\[
A_m \le C n^{qH-1/2} \|f\|_{m+a,p} n^{-H(q+2a+c)}.
\]
When $H>1/2$,
\begin{align*}
A_m &\le Cn^{qH-1/2} \|f\|_{m+a,p}n^{-a} n^{-H(c-b)} n^{(-H(q+b))\vee (1-(q+b))}
\\
&=Cn^{qH-1/2} \|f\|_{m+a,p}n^{-a-H(c-b)}n^{(-H(q+b))\vee (1-(q+b))}\\
&=C \|f\|_{m+a,p}n^{-a-1/2-H(c-b)}n^{(-Hb)\vee (q(H-1)+1-b)}.
\end{align*}

\medskip\noindent \emph{Case 2.} Suppose that $b=q$.
In this case, $c=q$ and $a=0$ and, applying Lemma \ref{betas}(a),
\begin{align*}
A_m&\le C n^{qH-1/2} \|f\|_{m,p} \!\left[\sum_{j,j'=0}^{n-1}|\beta_{j,k}|^q|\beta_{j',k}|^q\right]^{1/2}\\
&= Cn^{qH-1/2} \|f\|_{m,p} \sum_{j=0}^{n-1} |\beta_{j,k}|^q\\
&\le C \|f\|_{m,p}n^{qH-1/2} n^{(-2qH)\vee (1-2q)}.
\end{align*}
Note that $-2qH = (-2qH)\vee (1-2q)$.  Thus, this estimate coincides with the estimate in case 1 when $b=c=q$ and $a=0$.

\medskip\noindent \emph{Case 3.} Suppose that $b=0$.
Applying Lemma \ref{betas}(b) to $\sum_{j,j'=0}^{n-1} |\beta_{j,j'}|^q$ and Lemma \ref{lem1}(a) to each of the $\alpha$ terms,
\begin{align*}
A_m&\le  C n^{qH-1/2} \|f\|_{m+a,p} n^{-a(2H\wedge 1)}n^{-Hc} n^{(1/2-qH)\vee (1-q)}\\
&= C  n^{qH-1/2} \|f\|_{m+a,p}n^{-a(2H\wedge 1)}n^{-Hc} n^{1/2-qH} \\
&\le C \|f\|_{m+a,p}n^{-a(2H\wedge 1)-Hc}.
\end{align*}
This concludes the proof of the lemma.
\end{proof}
 \nocite{*}

{
\medskip
\noindent
{\bf Acknowledgement.} We would like to thank an anonymous  referee for his/her valuable comments.}

\end{document}